\documentclass[12pt]{article}
\usepackage{amsmath}
\usepackage{amsfonts}
\usepackage{amsthm}
\usepackage{latexsym}
\usepackage{amssymb}
\usepackage{verbatim}
\usepackage{enumerate}

\newtheorem{thrm}{Theorem}[section]

\newtheoremstyle{hdefinition}%
  {\topsep}%
  {\topsep}%
  {\upshape}
  {}%
  {\bfseries}%
  {.}
  { }%
  {\thmnumber{#2 }\thmname{#1}\thmnote{ \rm(#3)}}%
\newtheoremstyle{hclaim}%
  {\topsep}%
  {\topsep}%
  {\itshape}%
  {}%
  {\bfseries}%
  {.}
  { }%
  {\thmname{#1}\thmnote{ \rm#3}}%

\theoremstyle{hclaim}
\newtheorem*{claim*}{Claim}

\theoremstyle{hdefinition}

\theoremstyle{hclaim}

\begin{document}

\title{Ultrafilter limits of asymptotic density are not universally measurable}

\setcounter{footnote}{-1}
\author{J\"{o}rg Brendle \and Paul B. Larson\thanks{Supported by Invitation Fellowship for Research in Japan (Short-Term) ID No. S09138,
Japan Society for the Promotion of Science.
The first author was supported in part
by Grant-in-Aid for Scientific Research (C) 21540128,
Japan Society for the Promotion of Science,
and
the second author was supported in part by NSF Grant DMS-0801009. This note appeared in 
RIMS K\^{o}ky\^{u}roku Bessatsu No. 1686, 2009/11/16-2009/11/19, 16-18. }
}

\date{\empty}

\pagestyle{empty}

\maketitle

\thispagestyle{empty}

Given a nonprincipal ultrafilter $U$ on $\omega$ and a sequence
$\bar{x} = \langle x_{n} : n \in \omega \rangle$ consisting of members of a compact Hausdorff space
$X$, the $U$-\emph{limit} of $\bar{x}$ (written
$\lim_{n \to U}x_{n}$) is the unique $y \in X$ such that for every open set $O \subseteq X$ containing
$y$, $\{ n \mid x_{n} \in O\} \in U$. Letting $X$ be the unit interval $[0,1]$, this operation
defines a finitely additive measure $\mu_{U}$ on $\mathcal{P}(\omega)$ in terms of asymptotic density,
letting $\mu_{U}(A) = \lim_{n \to U} |A \cap n|/n$, for each $A \subseteq \omega$. A \emph{medial limit}
is a finitely additive measure on $\mathcal{P}(\omega)$, giving singletons measure $0$ and
$\omega$ itself measure $1$, such that for each open set $O \subseteq [0,1]$, the collection of
$A \subseteq \omega$ given measure in $O$ is universally measurable, i.e., is measured by every
complete finite Borel measure on $\mathcal{P}(\omega)$ (see \cite{L} for more on medial limits and universally
measurable sets). If there could consistently be a nonprincipal
ultrafilter $U$ such that measure given by the $U$-limit of asymptotic density were universally measurable,
this would give a relatively simple example of a medial limit. We show here, however, that this cannot be
the case.

\begin{thrm} If $U$ is a nonprincipal ultrafilter on $\omega$, then the function $$\mu_{U} \colon \mathcal{P}(\omega)
\to [0,1]$$ defined by letting $\mu_{U}(A) = \lim_{n \to U}|A \cap n|/n$ is not universally measurable.
\end{thrm}

\begin{proof} Let $I_{0} = \{0\}$, and for each positive $n \in \omega$
let $$I_{n} = \{ 5^{n-1}, 5^{n-1} + 1, \ldots, 5^{n} - 1\}.$$ Either the
union of the $I_{n}$'s for $n$ even is in $U$, or the corresponding
union for $n$ odd is. In the first case, let $J_{0} = I_{0} \cup I_{1} \cup I_{2}$,
and for each positive $n$, let $J_{n} = I_{2n+1} \cup I_{2n+2}$. In the second
case, let $J_{n} = I_{2n} \cup I_{2n+1}$ for all $n \in \omega$. In either case,
let $S$ be the set of $A \subseteq \omega$ such that $A \cap J_{n} \in \{\emptyset,
J_{n}\}$ for all $n \in \omega$. Then $S$ is a perfect subset of $\mathcal{P}(\omega)$,
and the mapping $H \colon S \to \mathcal{P}(\omega)$ sending $A \in S$ to $\{n \mid A \cap J_{n} = J_{n}\}$
is a homeomorphism.
Let $F_{0}$ be the set of $A \subseteq \omega$ such that $\mu_{U}(A) \in [0,1/4)$, and let
$F_{1}$ be the set of $A \subseteq \omega$ such that $\mu_{U}(A) \in (3/4, 1]$ (so $F_{1}$ is the set
of complements of elements of $F_{0}$).
It will be enough to show that $F_{1} \cap S$ is not a universally measurable subset of
$S$.

We claim for each $A \in S$, the set of $n$ such that $|A \cap n|/n \in [0,1/4)\cup(3/4,1]$ is in $U$.
This follows from the fact that all (but possibly one) of the $J_{n}$'s are unions of two
consecutive $I_{m}$'s, and that the union of the larger members of these pairs is in $U$.
Each such consecutive pair (for $n > 0$) has the form $I_{m} = \{ 5^{m-1}, 5^{m-1} + 1, \ldots, 5^{m} - 1\}$
and $I_{m+1} = \{ 5^{m}, 5^{m} + 1, \ldots, 5^{m+1} - 1\}$, and if $A \in S$, then $A$ either
contains or is disjoint from $I_{m} \cup I_{m+1}$. If it contains both, then for each
$k \in I_{m+1}$, $$|A \cap k|/k \geq (5^{m} - 5^{m-1})/5^{m} = 1 - 1/5 = 4/5 > 3/4,$$ and if it is
disjoint from both then $$|A \cap k|/k \leq 5^{m-1}/5^{m} = 1/5 < 1/4.$$ This establishes the claim.
It follows that $S \subseteq F_{0} \cup F_{1}$.
Since $\mu_{U}$ is a finitely
additive measure, the intersection of two sets of $\mu_{U}$-measure greater than $3/4$ cannot be less than $1/4$,
so $F_{1} \cap S$ is closed under finite intersections. It follows that $H$ maps $F_{1} \cap S$ homoemorphically
to a nonprincipal ultrafilter, and thus that $F_{1} \cap S$ is not universally measurable.
\end{proof}

A version of the proof just given, in the special case $\{ 5^{n} : n \in \omega \} \in U$, led to the proof in \cite{L} that
consistently there are no medial limits.

\begin{minipage}[t]{15em}
\noindent
J\"org Brendle
\\[0.5em]
Graduate School of Engineering\\ Kobe University\\ Rokko-dai 1-1, Nada\\ Kobe 657-8501 \\JAPAN\\
brendle@kurt.scitec.kobe-u.ac.jp
\end{minipage}
\hfill
\begin{minipage}[t]{15em}
\noindent
Paul B. Larson
\\[0.5em]
Department of Mathematics\\
Miami University\\
Oxford, Ohio 45056\\
USA\\
larsonpb@muohio.edu
\end{minipage}


\begin{thebibliography}{99}


\bibitem{L}
  P.B. Larson,
  \emph{The Filter Dichotomy and medial limits},
  Journal of Mathematical Logic 9 (2009) 2, 159-165 




\end{thebibliography}
\end{document}